\documentclass[12pt,a4paper]{article}
 \usepackage{graphicx}
\usepackage{amsfonts}
\usepackage{amssymb}
\usepackage{amsthm}
\usepackage[centertags]{amsmath}
\usepackage{newlfont}
\usepackage{enumitem}
\usepackage{color}
\usepackage{graphicx,color}
\usepackage{amsmath, amssymb, graphics}

\def\r{\mathbb R}
 
\def\t{\mathbf t}
\def\n{\mathbf n}

\def\b{\mathbf b}

\newtheorem{theorem}{Theorem}[section]

\newtheorem{definition}[theorem]{Definition}

\newtheorem{remark}[theorem]{Remark}
\newtheorem{proposition}[theorem]{Proposition}

\title{Classification and construction of minimal translation surfaces in Euclidean space  }
\author{Thomas Hasanis\\
Department of Mathematics\\
               University of Ioannina\\
               45110 Ioannina, Greece\\
\texttt{thasanis@cc.uoi.gr}\\
\and
Rafael L\'opez\footnote{Partially supported by the grant no. MTM2017-89677-P, MINECO/AEI/FEDER, UE.}\\
Departamento de Geometr\'{\i}a y Topolog\'{\i}a\\
 Instituto de Matem\'aticas (IEMath-GR)\\
 Universidad de Granada\\
 18071 Granada, Spain\\
\texttt{rcamino@ugr.es}
}

\date{}
\begin{document}
\maketitle

\begin{abstract}
A  translation surface of Euclidean space $\r^3$ is the sum of two regular curves $\alpha$ and $\beta$, called the generating curves. In this paper we classify the minimal translation surfaces of $\r^3$ and we give a method of construction of explicit examples. Besides the plane and the minimal surfaces of Scherk type, it is proved that up to reparameterizations of the generating curves, any minimal translation surface is   described as $\Psi(s,t)=\alpha(s)+\alpha(t)$, where  $\alpha$  is a curve parameterized by arc length $s$, its curvature $\kappa$   is a positive solution of the   autonomous ODE $(y')^2+y^4+c_3y^2+c_1^2y^{-2}+c_1c_2=0$  and its torsion is $\tau(s)=c_1/\kappa(s)^2$. Here  $c_1\not=0$, $c_2$ and $c_3$ are constants  such that the cubic equation $-\lambda^3+c_2\lambda^2-c_3\lambda+c_1=0$ has three real roots $\lambda_1$, $\lambda_2$ and $\lambda_3$.

\end{abstract}


\section{Introduction and statement of the result}

 The surfaces of our study have its origin in the classical text of G. Darboux \cite[Livre I]{da} where the so-called ``surfaces d\'efinies pour des properti\'et\'es cin\'ematiques'' are considered, and later known as Darboux surfaces in the literature. A {\it Darboux surface} is defined kinematically as the movement of a curve by a uniparametric family of rigid motions of $\r^3$. Hence a parameterization of a such surface is $\Psi(s,t)=A(t)\cdot \alpha(s)+\beta(t)$, where $\alpha$ and $\beta$ are two space  curves and $A(t)$ is an orthogonal matrix. The case that we are investigating in this paper is that   $A(t)$ is the identity. More precisely, we give the following definition.

\begin{definition} A surface $S\subset\r^3$ is called a translation surface if $S$  can be locally written as the sum $\Psi(s,t)=\alpha(s)+\beta(t)$ of two space curves $\alpha:I\subset\r\rightarrow\r^3$ and $\beta:J\subset\r\rightarrow\r^3$. The curves $\alpha$ and $\beta$ are called the generating curves of $S$.  If $\alpha$ and $\beta$ are plane curves, the surface is called a translation surface of planar type.  \end{definition}

Darboux deals with translation surfaces in Secs. 81-84 \cite[pp. 137--142]{da} and its name is due to the fact that   the surface $S$  is obtained by the translation of $\alpha$ along $\beta$ (or {\it vice versa} because the roles of $\alpha$ and $\beta$ are interchanged). In particular,  all parametric curves $s=ct$ are congruent by translations (similarily for parametric curves $t=ct$).
It is natural to ask for the classification of the translation surfaces of $\r^3$ under some condition on its curvature. Recently, the authors of the present paper succeeded  with the complete classification of all translation surfaces with constant Gaussian curvature $K$, proving that  the only ones are   cylindrical surfaces and thus,   $K$ must be $0$ (\cite{hl}).

In this paper we are concerned with the following

\begin{quote}{\bf Problem:} Classify all minimal translation surfaces in $\r^3$.
\end{quote}

Recall that a minimal surface in $\r^3$ is a surface with zero mean curvature everywhere.   Of course, the plane is a trivial example of a minimal translation surface. A first approach to the posed  problem  is to assume that the  generating curves are   plane curves contained in orthogonal planes. In such a case, after an appropriate choice of coordinate system,   the     surface $S$ is locally parameterized by
$$\Psi(s,t)=(s,0,f(s))+(0,t,g(t))=(s,t,f(s)+g(t))$$
for some two smooth functions $f$ and $g$. Thus the problem transforms into finding  surfaces of the form $z=f(x)+g(y)$ with zero mean curvature. It is not difficult to see that, besides the plane and a  rigid motion, the solution is  the known Scherk surface
$$z=\frac{1}{c}\log\frac{\cos(cy)}{\cos(cx)},\quad x,y\in\left(\frac{-\pi}{2},\frac{\pi}{2}\right),$$
where $c$ is a positive constant. This surface was obtained by Scherk solving the minimal surface equation by separation of variables, namely, $z=f(x)+g(y)$ (\cite{sc}). In fact, this  surface  belongs to a uniparametric family of minimal surfaces discovered by Scherk and given by
$$\mathcal{S}_\theta(x,y)=\left(x+y\cos\theta,y\sin\theta,\frac{1}{c}\log\frac{\cos(cy)}{\cos(cx)}\right),$$
where  $\theta\in\r$ (\cite{sc}; see also \cite[\S 81]{ni}). For $\theta=0$, $\mathcal{S}_0$ is the plane and if $\theta=\pi/2$, $\mathcal{S}_{\pi/2}$ is the Scherk surface. Let us observe that  $\mathcal{S}_\theta$ is a translation surface where  the generating curves are planar but now not necessarily contained in orthogonal planes. Indeed, $\mathcal{S}_\theta$ can be expressed  as
\begin{eqnarray*}
\mathcal{S}_\theta(s,t)&=&\left(s,0,-\frac{1}{c}\log(\cos(cs))\right)+\left(t\cos\theta,t\sin\theta,\frac{1}{c}\log(\cos(ct))\right)\\
&=&\alpha(s)+\beta(t).
\end{eqnarray*}
Then $\alpha$ is contained in the $xz$-plane and $\beta$ lies  in the plane of equation $\sin\theta x-\cos\theta y=0$, which makes an angle $\theta$ with the $xz$-plane.

Other minimal translation surface, and already known by Lie, is the helicoid $X(u,v)=(\cos u\cos v,\sin u\cos v,u)$ (\cite[\S 77]{ni}). This surface is obtained by the sum of a circular helix with   itself. Indeed, let  $\alpha$ be the circular helix parameterized as $\alpha(s)=(\cos s,\sin s,s)/2$. Making the change of coordinates $s=u+v$, $t=u-v$, we find $\Psi(s,t)= \alpha(s)+\alpha(t)=X(u,v)$.

In the literature, other works have appeared on the study of translation surfaces with constant mean curvature, also in other ambient spaces: we refer \cite{Dillen98,HLin,Lopez,lomu,mm,mp}, without to be a complete list. However, in all these works, the translation surface is of planar type, so  the problem of finding such surfaces reduces into a problem of solving a PDE by separation of variables. It deserves to point out  that it was proved   in  \cite{Dillen98} that if one generating curve is planar, then the other is also planar and the surface belongs to the family of Scherk surfaces.

Definitively, besides the plane, the Scherk  surfaces and the helicoid, the motivating question for  the  present paper is as follows:

\begin{quote}
{\bf Question:} Are there any other minimal translation surfaces in $\r^3$?\end{quote}

Recently the second author   and O. Perdomo have undertaken the problem of classification in all its generality assuming that the generating curves are space curves (\cite{lp}). It was proved that  one generating curve is the rigid motion of the other one, hence we can write $\Psi(s,t)=\alpha(s)+\alpha(t)$ and if $\kappa$ and $\tau$ are the curvature  and the  torsion of $\alpha$ respectively, then  $\kappa^2\tau$ is a non-zero constant.  Furthermore, the velocity vectors  $\alpha'(s)$ must lie in a cone of the form $ \{{\mathbf x}\in\r^3: \langle A{\mathbf x},{\mathbf x}\rangle=0 \}$ for a particular symmetric matrix $A$: see \cite{lp} for details.

In this paper we offer an alternative approach to the study of the minimal translation surfaces. Besides to obtain similar results than the ones of \cite{lp}, we give a new method of construction of  minimal translation surfaces based on the resolution of an ODE   which seems to us   simpler than the methods used in \cite{lp}. The advantage  lies in the fact that provides a technique by means of a recipe that gives   a plethora of examples. The results may be summarized as follows.

\begin{theorem}[classification and construction]\label{t1} Let $\Psi(s,t)=\alpha(s)+\beta(t)$ be a minimal translation surface where $\alpha$ and $\beta$ parameterized by arc length. Suppose that the curvatures $\kappa_\alpha$ and $\kappa_\beta$ are positive everywhere and the torsions $\tau_\alpha$, $\tau_\beta$ are non-zero everywhere. Then, up to a rigid motion, the curve $\beta$ coincides with $\alpha$  and
\begin{equation}\label{eqcc}
\kappa_\alpha^2\tau_\alpha=c_1,\quad \frac{1}{\tau_\alpha}\left(\frac{\kappa_\alpha'}{\kappa_\alpha}\right)'+\frac{\kappa_\alpha^2}{\tau_\alpha}-2\tau_\alpha=c_2,
\end{equation}
where $c_1\not=0$ and $c_2$ are two constants. Furthermore,  $\kappa_\alpha$ is a positive solution of the autonomous ODE
\begin{equation}\label{eqt2}
y'^2+y^4+c_3y^2+\frac{c_1^2}{y^{2}}+c_1c_2=0
\end{equation}
for some constant $c_3$ and the curve $\alpha$ can be expressed as
$$\alpha(s)=\left(A\int^s\cos w(s),B\int^s\sin w(s),\int^s\sqrt{1-A^2\cos^2 w(s)-B^2\sin^2w(s)}\right).$$
Here  $A=\sqrt{\lambda_3/(\lambda_3-\lambda_1)}$ and $B=\sqrt{ \lambda_3/(\lambda_3-\lambda_2)}$ where $\lambda_i$, $1\leq i\leq 3$,   are the three real roots of the cubic equation
\begin{equation}\label{eqt1}
-\lambda^3+c_2\lambda^2-c_3\lambda+c_1=0
\end{equation}
 and the function $w$ is    $w(s)=\int^s\sqrt{\kappa_\alpha(s)^2+\lambda_1\lambda_2}$.

 Conversely, any minimal translation surface of $\r^3$ of non-planar type is constructed by this process. Exactly  let  $c_1\not=0$, $c_2$ and $c_3$ be three constants such that the cubic equation (\ref{eqt1})
has three real roots $\lambda_1$, $\lambda_2$, $\lambda_3$. Let  $\kappa(s)$ be a positive and non-constant solution of (\ref{eqt2}). If $\alpha(s)$ is a curve   parameterized by arc length $s$ with curvature $\kappa(s)$ and torsion $c_1/\kappa(s)^2$, then the translation surface $\Psi(s,t)=\alpha(s)+\alpha(t)$ is minimal.
\end{theorem}

This paper is organized as follows. In Section \ref{se2} we recall some known formulae of the local theory of curves and surfaces in $\r^3$.  In Section \ref{se3} we prove, for the sake of completeness, some known results with alternative proofs. So,   we prove the   result of \cite{Dillen98} and we obtain the helicoid when one generating curve is a circular helix (Theorem \ref{t32}).  We also characterize any minimal translation surface   by the two relations (\ref{eqcc}) between the curvature   and the torsion  of the generating curves. In Section \ref{se4}, we show the main results of this paper. Here it will be essential the definition  of a set of  self-adjoint linear operators on $\r^3$ associated to each point $\alpha(s)$ and $\beta(t)$, which  it will be proved later that, indeed, they coincide for all $s$ and $t$.  The two   results of this section   (Theorems \ref{t44} and \ref{t45}) classify and  describe the construction of all  minimal translation surfaces in $\r^3$.  The section finishes  showing explicit examples of  translation minimal surfaces by the  procedure previously proved (see also Theorem \ref{t1}).

\section{Preliminaries}\label{se2}
For a general reference on curves and surfaces, we refer to \cite{mr}. All the  curves and surfaces considered in this paper will be assumed to be of class $C^\infty$. Let $\alpha(s)$, $s\in I$, and $\beta(t)$, $t\in J$, be two curves in $\r^3$ parameterized by arc length with   oriented Frenet trihedrons $\{\t_\alpha(s),\n_\alpha(s),\b_\alpha(s)\}$, $\{\t_\beta(t),\n_\beta(t),\b_\beta(t)\}$, for every $s\in I$, $t\in J$, respectively. Throughout this paper let   $\kappa_\alpha(s)>0$ and  $\kappa_\beta(t)>0$ denote the curvatures of $\alpha$ and $\beta$ respectively, as well as,  $\tau_\alpha(s)$ and $\tau_\beta(t)$ the torsions. Let    $\{\alpha(s)+\beta(t): s\in I, t\in J\}\subset\r^3$ be the set obtained by the sum of the curves $\alpha$ and $\beta$. Then $S$ is a regular (translation) surface, and  $\Psi(s,t)=\alpha(s)+\beta(t)$ is a parameterization of $S$, if  $\t_\alpha(s)\times\t_\beta(t)\not=0$ for all $(s,t)\in I\times J$, where $\times$ represents the vector product of $\r^3$: throughout the paper, we will make this assumption.    Recall that the parametric curves $t=ct$ are congruent and translations of $
\alpha(s)$. Hence, they have the same curvature and torsion at corresponding points (similarly for the parametric curves $s=ct$).

We calculate the Gauss curvature and the mean curvature of $S$. For notational convenience we omit the dependence on $s$ and $t$ of the function which are implicitly understood. The derivatives of order $1$ of $\Psi$ are $\Psi_s=\t_\alpha$, $\Psi_t=\t_\beta$, with  $\Psi_s\times\Psi_t\not=0$. Let $\phi(s,t)$, $0<\phi(s,t)<\pi$, be the angle that $\t_\alpha(s)$ makes with $\t_\beta(t)$ at point $\Psi(z,t)$, that is,
$$\cos\phi(s,t)=\langle\t_\alpha(s),\t_\beta(t)\rangle,$$
where $\langle,\rangle$ stands for the usual scalar product of $\r^3$. Then the  coefficients of the first fundamental form in the basis $\{\Psi_s,\Psi_t\}$ are
$$E=1,\ F=\cos\phi,\ G=1,$$
and  the unit normal vector $N(s,t)$  at $\Psi(s,t)$ is
$$N(s,t)=\frac{\t_\alpha(s)\times\t_\beta(t)}{\sin\phi(s,t)}.$$
The derivatives of $\Psi$ of order $2$ are $\Psi_{ss}=\kappa_\alpha\n_\alpha$, $\Psi_{st}=0$ and $\Psi_{tt}=\kappa_\beta\n_\beta$, hence the coefficients of the second fundamental form are
$$l=-\frac{\kappa_\alpha}{\sin\phi}\langle\b_\alpha,\t_\beta\rangle,\ m=0,\ n=\frac{\kappa_\beta}{\sin\phi}\langle\t_\alpha,\b_\beta\rangle.$$
The Gaussian curvature $K$ and the mean curvature $H$ of $S$ are
\begin{equation}\label{eqk}
K=-\frac{\kappa_\alpha\kappa_\beta}{\sin^4\phi}\langle\b_\alpha,\t_\beta\rangle\langle\t_\alpha,\b_\beta\rangle,\quad H=\frac{-\kappa_\alpha\langle\b_\alpha,\t_\beta\rangle+\kappa_\beta\langle t_\alpha,\b_\beta\rangle}{2\sin^3\phi}.
\end{equation}
Consequently,   $S$ is a minimal surface ($H=0$) if and only if
\begin{equation}\label{eq1}
\kappa_\alpha\langle\b_\alpha,\t_\beta\rangle=\kappa_\beta\langle\t_\alpha,\b_\beta\rangle\quad \mbox{for all $s\in I$, $t\in J.$}
\end{equation}

\begin{remark} If the generating curves $\alpha$ and  $\beta$ are not parameterized by arc length, then the minimality condition $H=0$  is equivalent to
\begin{equation}\label{eq2}
|\beta'(t)|^2\langle\alpha'(s)\times\alpha''(s),\beta'(t)\rangle=|\alpha'(s)|^2\langle\alpha'(s),\beta'(t)\times\beta''(t)\rangle
\end{equation}
for all $s\in I$, $t\in J$.
\end{remark}

The following curve will be useful later,
\begin{equation}\label{eq22}
\alpha(u)=\left(u,0,-\frac{1}{c}\log(\cos (cu))\right),\ u\in\left(\frac{-\pi}{2c},\frac{\pi}{2c}\right),
\end{equation}
where $c$ is a positive constant, which is nothing more  than the generating curve of the Scherk surface $\mathcal{S}_{\pi/2}$. Its curvature, with parameter the arc length $s$ and origin $u=0$, is
\begin{equation}\label{eq233}
\kappa_\alpha(s)=\frac{2c e^{cs}}{1+e^{2cs}}.
\end{equation}

\section{Auxiliary results}\label{se3}

In this section we    characterize any minimal translation surface   by the two relations   between the curvature   and the torsion  of the generating curves. 
 and  we prove  some known results with alternative proofs. Let $S\subset\r^3$ be a minimal translation surface with parameterization $\Psi(s,t)=\alpha(s)+\beta(t)$ where the generating curves $\alpha$ and $\beta$ are parameterized by arc length $s$ and $t$.  Having in mind the Frenet equations, we take the derivative with respect to $s$ of (\ref{eq1}) and then dividing by $\kappa_\alpha$, we arrive at
\begin{equation}\label{eq4}
\langle-\tau_\alpha\n_\alpha+\frac{\kappa_\alpha'}{\kappa_\alpha}\b_\alpha,\t_\beta\rangle=\langle\n_\alpha,\kappa_\beta\b_\beta\rangle.
\end{equation}
Differentiating (\ref{eq4}) with respect to   $s$  again and taking into account (\ref{eq1}) and (\ref{eq4}), we obtain
\begin{equation}\label{eq5}
\langle \kappa_\alpha\tau_\alpha\t_\alpha-\left(\frac{\kappa_\alpha'}{\kappa_\alpha}\tau_\alpha+\tau_\alpha'\right)\n_\alpha+\left(\left(\frac{\kappa_\alpha'}{\kappa_\alpha}\right)'+\kappa_\alpha^2-\tau_\alpha^2\right)\b_\alpha,\t_\beta\rangle=
\tau_\alpha\langle\b_\alpha,\kappa_\beta\b_\beta\rangle.
\end{equation}

In the same way, for the curve $\beta$ we have
\begin{equation}\label{eq7}
\tau_\beta\langle\kappa_\alpha\b_\alpha,\b_\beta\rangle=\langle\t_\alpha,\kappa_\beta\tau_\beta\t_\beta-
\left(\frac{\kappa_\beta'}{\kappa_\beta}\tau_\beta+\tau_\beta'\right)\n_\beta+\left(\left(\frac{\kappa_\beta'}{\kappa_\beta}\right)'+\kappa_\beta^2-\tau_\beta^2\right)\b_\beta\rangle.
\end{equation}
Once obtained the above formulae, and for the completeness of this work, we insert in this section the result proved in \cite{Dillen98} with  an alternative proof.

\begin{proposition}\label{pr-dillen} Let $S$ be a non planar minimal translation surface. Assume   that  one, say $\alpha$, of the generating curves is a plane curve. Then:
\begin{enumerate}
\item The curvature $\kappa_\alpha$ of $\alpha$ satisfies the autonomous ODE
\begin{equation}\label{eqyy}
\left(\frac{y'}{y}\right)'+y^2=0.
\end{equation}
\item The curve $\alpha$ is a rigid motion of the curve (\ref{eq22}).
\item The other generating curve $\beta$ is also a plane curve and  $S$ is a surface of Scherk type.
\end{enumerate}
\end{proposition}

\begin{proof}
\begin{enumerate}
\item From (\ref{eq5}), because of $\tau_\alpha=0$, we have
$$\left(\left(\frac{\kappa_\alpha'}{\kappa_\alpha}\right)'+\kappa_\alpha^2\right)\langle \b_\alpha,\t_\beta\rangle=0$$
 for all $s$ and $t$. If $\langle\b_\alpha,\t_\beta\rangle=0$ on an open set, then from (\ref{eqk}) we have $K=0$. Since $H=K=0$, then  $S$ is a plane, a contradiction. So, we must have
$$\left(\frac{\kappa_\alpha'}{\kappa_\alpha}\right)'+\kappa_\alpha^2=0.$$
\item The general solution of the autonomous ODE  (\ref{eqyy}) is
$$\kappa_\alpha(t)=\frac{2c e^{\pm (ct+c_1)}}{1+ e^{\pm2(ct+c_1)}},$$
where $c>0$, $c_1$ are constants. After the change $s=\pm (ct+c_1)$, we see that the curvature of $\alpha$ is the same than (\ref{eq233}). From   the fundamental theorem of plane curves, the curve $\alpha$ is a rigid motion of the curve (\ref{eq22}). 
\item After a rigid motion, we may suppose that $\alpha$ is as in (\ref{eq22}). Let   $\beta(v)=(\beta_1(v),\beta_2(v),\beta_3(v))$ denote the  other generating curve   parameterized by arc length $v$. Then  the minimality condition (\ref{eq2}) gives
$$(\beta_1'\beta_2''-\beta_1''\beta_2')\sin(cu)+(c\beta_2'+\beta_2'\beta_3''-\beta_2''\beta_3')\cos(cu)=0.$$
Since the functions $\sin(cu)$ and $\cos(cu)$ are linearly independent, we deduce
\begin{equation}\label{eqbb}
\beta_1'\beta_2''-\beta_1''\beta_2'=0,\quad c\beta_2'+\beta_2'\beta_3''-\beta_2''\beta_3'=0.
\end{equation}
Combining both equations, we have $\beta_2'(\beta_1''\beta_3'-c\beta_1'-\beta_1'\beta_3'')=0$. If $\beta_2$ is a constant function, then $\beta$ is a plane curve. On the contrary, $\beta_1''\beta_3'=c\beta_1'+\beta_1'\beta_3''$. From the first equation in (\ref{eqbb}), we obtain $\beta_1'\beta_2'''=\beta_1'''\beta_2'$. Then   the mixed product $(\beta',\beta'',\beta''')$ is
$$(\beta',\beta'',\beta''')=-c\beta_1'''\beta_2'-\beta_2'''(\beta_1'\beta_3''-\beta_1''\beta_3')=\beta_2'''(c\beta_1'+\beta_1'\beta_3''-\beta_1''\beta_3')=0.$$
This implies that $\tau_\beta=0$ and  $\beta$  is planar. Now, according to the item 2  of the proposition,   $\beta$ is, up to a rigid motion, the curve parameterized in (\ref{eq22}). Set  $\beta(v)=A\sigma(v)$, where $A$ is an orthogonal matrix, $\sigma(v)=(v,0,-\frac{1}{d}\log\cos(dv))$ and $d>0$ is a positive constant. Applying the minimality condition (\ref{eq2}) again, we have
$$c\left(a_{21}+a_{23}\frac{\sin (dv)}{\cos(dv)}\right)=d \left(a_{23}a_{31}-a_{21}a_{33}+(a_{13} a_{21}-a_{11}a_{23})\frac{\sin(cu)}{\cos(cu)}\right).$$
Due to the linear independence, in the first step, of $\cos(cu)$, $\sin(cu)$, and then of $\cos(dv)$, $\sin(dv)$, we deduce $a_{23}=0$ and
$$a_{13}a_{21}=0\quad ca_{21}+da_{21}a_{33}=0.$$
In case $a_{21}=0$, and using that $A$ is orthogonal, it follows that  $a_{22}=\pm 1$ and it is not difficult to see that $\beta$ is a curve contained in the $xz$-plane, the same that $\alpha$: this implies that $S$ is a plane, a contradiction. Thus $a_{21}\not=0$ and $a_{13}=0$ and $c=-d a_{33}$. Using that $A$ is orthogonal, then $a_{31}=a_{32}=0$ and $a_{33}=\pm 1$. In particular, and because $c$ and $d$ are positive,  we find  $a_{33}=-1$. Definitively, we have two possibilities for the matrix $A$, namely,
$$ \left(\begin{array}{lll}\cos\theta&\sin\theta&0\\ \sin\theta&-\cos\theta&0\\ 0&0& -1\end{array}\right)\mbox{ and }  \left(\begin{array}{lll}\cos\theta&-\sin\theta&0\\ \sin\theta&\cos\theta&0\\ 0&0& -1\end{array}\right).$$
In both cases, the parameterization $\Psi(s,t)$ is
$$\Psi(u,v)=\alpha(u)+A\sigma(v)=\left(u+v\cos\theta,v\sin\theta,\frac{1}{c}\log\frac{\cos(cv)}{\cos(cu)}\right),$$
and $S$ is the surface $\mathcal{S}_\theta$ belonging to the Scherk family.
\end{enumerate}
\end{proof}

Recall that the helicoid is a minimal translation surface obtained  as the sum of a circular helix with itself. We prove that this is consequence of the following result (see also \cite[Cor. 3.4]{lp}).

\begin{theorem}\label{t32} Let $S$ be a minimal translation surface. If one of the generating curves is a circular helix, then the other curve is a congruent circular helix and $S$ is the helicoid.
\end{theorem}

\begin{proof} Assume that the generating curve $\alpha$ of $S$ is the circular helix
$$\alpha(s)=\left(a\cos\varphi(s),a\sin\varphi(s),b\varphi(s)\right),$$
where $\varphi(s)=s/\sqrt{a^2+b^2}$, $a>0$, $b\not=0$ are two constants.  Then
$\kappa_\alpha=a/(a^2+b^2)$, $\tau_\alpha=b/(a^2+b^2)$
and
$$\t_\alpha(s)=\frac{1}{\sqrt{a^2+b^2}}\left(-a\sin\varphi(s),a\cos\varphi(s),b\right)$$
$$\b_\alpha(s)=\frac{1}{\sqrt{a^2+b^2}}\left(b\sin\varphi(s),-b\cos\varphi(s),a\right).$$
If $\beta(t)=(\beta_1(t),\beta_2(t),\beta_3(t))$ is the other generating curve parameterized by arc length $t$, then $\kappa_\beta=|\beta''|=|\beta'\times\beta''|$, $\t_\beta=(\beta_1',\beta_2',\beta_3')$ and
$$\b_\beta=\frac{1}{|\beta'\times\beta''|}
((\beta'\times\beta'')_1,(\beta'\times\beta'')_2,(\beta'\times\beta'')_3).$$
Applying the minimality condition (\ref{eq1}), it follows
or equivalently,
\begin{eqnarray*}&&\frac{a}{a^2+b^2}\left(\beta_1'b\sin \varphi(s)-\beta_2'b\cos\varphi(s)+a\beta_3'\right)\\
&&=\
\left(-a(\beta'\times\beta'')_1\sin\varphi(s)+a(\beta'\times\beta'')_2\cos\varphi(s)+b(\beta'\times\beta'')_3\right).
\end{eqnarray*}
Since the functions $\{1, \cos\varphi(s),\sin\varphi(s)\}$ are linearly independent, we derive
\begin{equation}\label{bbb}
\begin{split}
\frac{b}{a^2+b^2}\beta_1'&=-(\beta'\times\beta'')_1=-\beta_2'\beta_3''+\beta_2''\beta_3'\\
\frac{b}{a^2+b^2}\beta_2'&=-(\beta'\times\beta'')_2=-\beta_3'\beta_1''+\beta_3''\beta_1'\\
\frac{a^2}{a^2+b^2}\beta_3'&=b(\beta'\times\beta'')_3=b(\beta_1'\beta_2''-\beta_1''\beta_2').
\end{split}
\end{equation}
Multiplying the first and second of (\ref{bbb}) by $\beta_1'$ and $\beta_2'$ respectively, we deduce
$$\frac{b}{a^2+b^2}(\beta_1'^2+\beta_2'^2)=\beta_3'(\beta_1'\beta_2''-\beta_1''\beta_2')=\frac{a^2 \beta_3'^2}{b(a^2+b^2)},$$
where in the last identity we have used the third of (\ref{bbb}).  Since $\beta_1'^2+\beta_2'^2=1-\beta_3'^2$, then  $\beta_3'^2=b^2/(a^2+b^2)$, hence 
$$\beta_1'^2+\beta_2'^2=\frac{a^2}{a^2+b^2}.$$
Without loss of generality, we take $\beta_3'=b/\sqrt{a^2+b^2}$. Then  the indicatrix of tangents of $\alpha$ and $\beta$  will be lying to the same hemisphere. Then we must have
$$\beta_1'(t)=-\frac{a}{\sqrt{a^2+b^2}}\sin\varphi(t),\ \beta_2'=\frac{a}{\sqrt{a^2+b^2}}\cos\varphi(t)$$
and thus, up to a translation,  
$\beta(t)=\left(a\cos\varphi(t),a\sin\varphi(t),b\varphi(t)\right)$ and $\beta$  coincides with   $\alpha$.
\end{proof}
As a consequence of Proposition \ref{pr-dillen}, from now on we may suppose that the generating curves $\alpha$ and $\beta$ are non planar. We need to introduce the following notation. For a non plane curve parameterized by arc length  with curvature $\kappa$ and torsion $\tau$, we set
\begin{equation}\label{eqR}
R=\frac{\kappa'}{\kappa}+\frac{\tau'}{\tau},\quad \Sigma=\left(\frac{\kappa'}{\kappa}\right)'+\kappa^2-\tau^2.
\end{equation}
The subscript $\alpha$ or $\beta$ in $R$ and $\Sigma$ indicates that we are working in the corresponding curve $\alpha$ or $\beta$.

We now have the following key result.

\begin{proposition}\label{pr34}
 If $\Psi(s,t)=\alpha(s)+\beta(t)$ is a minimal translation surface, then
\begin{equation}\label{eq88}
\kappa_\alpha^2\tau_\alpha=c_1\not=0,\ \kappa_\beta^2\tau_\beta=\bar{c}_1\not=0,\ \frac{\Sigma_\alpha}{\tau_\alpha}-\tau_\alpha=c_2,\ \frac{\Sigma_\beta}{\tau_\beta}-\tau_\beta=\bar{c}_2,
\end{equation}
 where $c_1$, $c_2$, $\bar{c}_1$ and $\bar{c}_2$ are constants.
 \end{proposition}

 \begin{proof}
 Dividing (\ref{eq5}) by $\tau_\alpha$, we have
$$
 \langle\kappa_\alpha\t_\alpha-R_\alpha\n_\alpha+\frac{\Sigma_\alpha}{\tau_\alpha}\b_\alpha,\t_\beta\rangle=\langle\b_\alpha,\kappa_\beta\b_\beta\rangle.
$$
By differentiation this equation  with respect to $s$ and taking into account the Frenet equations and (\ref{eq4}), we arrive at
 \begin{equation}\label{eq10}
 \langle u,\t_\beta\rangle=0,
 \end{equation}
 where we have set
 \begin{equation}\label{eq11}
 u=(\kappa_\alpha R_\alpha+\kappa'_\alpha)\t_\alpha+(\kappa_\alpha^2-\tau_\alpha^2-\Sigma_\alpha-R_\alpha')\n_\alpha+\left(\left(\frac{\Sigma_\alpha}{\tau_\alpha}\right)'-\tau_\alpha R_\alpha+\frac{\tau_\alpha}{\kappa_\alpha}\kappa'_\alpha\right)\b_\alpha.
\end{equation}
Differentiating  (\ref{eq10}) with respect to  $t$ and because $\kappa_\beta>0$, we have
\begin{equation}\label{eq16}
\langle u,\n_\beta\rangle=0.
\end{equation}
Finally, we differentiate again (\ref{eq16}) with respect to $t$ and, taking into account (\ref{eq10}) and $\tau_\beta\not=0$, we have
\begin{equation}\label{eq17}
\langle u,\b_\beta\rangle=0.
\end{equation}
From equations (\ref{eq10}), (\ref{eq16}) and (\ref{eq17}), we find $u=0$. Hence we deduce from (\ref{eq11})
\begin{equation} \label{eq18}
\left.
\begin{split}
\kappa_\alpha R_\alpha+\kappa'_\alpha&=0\\
\Sigma_\alpha+R_\alpha'-\kappa_\alpha^2+\tau_\alpha^2&=0\\
\left(\frac{\Sigma_\alpha}{\tau_\alpha}\right)'-\tau_\alpha R_\alpha+\frac{\kappa'_\alpha}{\kappa_\alpha}\tau_\alpha&=0.
\end{split}\right\}
\end{equation}
The first of (\ref{eq18}) implies
\begin{equation}\label{eq19}
R_\alpha=-\frac{\kappa'_\alpha}{\kappa_\alpha}.
\end{equation}
By the definition of $R_\alpha$ in (\ref{eqR}), we derive  $2\kappa_\alpha'\tau_\alpha+\kappa_\alpha\tau_\alpha=0$. Then there exists a constant  $c_1\not=0$ such that $\kappa_\alpha^2\tau_\alpha=c_1$.  By the definition of $\Sigma_\alpha$ in (\ref{eqR}) , the second relation of (\ref{eq18}) is $(\kappa_\alpha'/\kappa_\alpha)'+R_\alpha'=0$, which is valid because of (\ref{eq19}). The third of (\ref{eq18}) and the definition of $R_\alpha$ give $(\Sigma_\alpha/\tau_\alpha)'-\tau_\alpha'=0$, hence
$$\frac{\Sigma_\alpha}{\tau_\alpha}-\tau_\alpha=c_2
$$for some constant $c_2$. In a similar way, we deduce the corresponding results for the curve $\beta$ by using (\ref{eq7}).
\end{proof}

\begin{remark} In conclusion, with successive differentiations of (\ref{eq1}) with respect to $s$, $ss$, $t$, $tt$, $ts$, $sst$, $tts$ and $ttss$ we, respectively, find

\begin{equation}\label{eq23}
\left.
\begin{split}
\langle-\tau_\alpha\n_\alpha-R_\alpha\b_\alpha,\t_\beta\rangle&=\langle\n_\alpha,\kappa_\beta\t_\beta\rangle\\
 \langle\kappa_\alpha\t_\alpha-R_\alpha\n_\alpha+\frac{\Sigma_\alpha}{\tau_\alpha}\b_\alpha,\t_\beta\rangle&=\langle\b_\alpha,\kappa_\beta\b_\beta\rangle\\
 \langle\kappa_\alpha\b_\alpha,\n_\beta\rangle&=\langle\t_\alpha,-\tau_\beta\n_\beta-R_\beta\b_\beta\rangle\\
 \langle\kappa_\alpha\b_\alpha,\b_\beta\rangle&=\langle\t_\alpha,\kappa_\beta\t_\beta-R_\beta\n_\beta+\frac{\Sigma_\beta}{\tau_\beta}\b_\beta\rangle\\
 \langle-\tau_\alpha\n_\alpha-R_\alpha\b_\alpha,\n_\beta\rangle&=\langle\n_\alpha,-\tau_\beta\n_\beta-R_\beta\b_\beta\rangle\\
\langle \kappa_\alpha\t_\alpha-R_\alpha\n_\alpha+\frac{\Sigma_\alpha}{\tau_\alpha}\b_\alpha,\n_\beta\rangle&=\langle\b_\alpha,-\tau_\beta\n_\beta-R_\beta\b_\beta\rangle\\
 \langle-\tau_\alpha\n_\alpha-R_\alpha\b_\alpha,\b_\beta\rangle&=\langle\n_\alpha,\kappa_\beta\t_\beta-R_\beta\n_\beta+\frac{\Sigma_\beta}{\tau_\beta}\b_\beta\rangle\\
 \langle\kappa_\alpha\t_\alpha-R_\alpha\n_\alpha+\frac{\Sigma_\alpha}{\tau_\alpha}\b_\alpha,\b_\beta\rangle&=\langle \b_\alpha,\kappa_\beta\t_\beta-R_\beta\n_\beta+\frac{\Sigma_\beta}{\tau_\beta}\b_\beta\rangle.
 \end{split}\right\}
\end{equation}
\end{remark}

Another useful result is the following.

\begin{proposition}\label{pr36}
 Let $\alpha$ be a curve in $\r^3$ parameterized by arc length with curvature $\kappa_\alpha>0$ and torsion $\tau_\alpha\not=0$. If $\sigma_1\not=0$ and $\sigma_2$ are two constants such that
\begin{equation}\label{eq31}
\kappa_\alpha^2\tau_\alpha=\sigma_1\not=0\mbox{ and }\frac{\Sigma_\alpha}{\tau_\alpha}-\tau_\alpha=\sigma_2,
\end{equation}
then    $\kappa_\alpha$ is a positive solution of the autonomous ODE
\begin{equation}\label{eq32}
y'^2+y^4+\sigma_3 y^2+\frac{\sigma_1^2}{y^{2}}+\sigma_1\sigma_2=0
\end{equation}
for some constant $\sigma_3$.

Conversely, let $c_1\not=0$, $c_2$ be two constant. Then for any positive and non-constant solution $\kappa(s)$ of (\ref{eq31}) and choosing $\tau(s)=\sigma_1/\kappa(s)^2$,  a curve $\alpha$ parameterized by the arc length $s$ with curvature and torsion $\kappa$ and $\tau$, respectively,  satisfies
\begin{equation}\label{eq33}
\frac{\Sigma_\alpha}{\tau_\alpha}-\tau_\alpha=\sigma_2,\quad\Sigma_\alpha+R_\alpha^2+\kappa_\alpha^2=-\sigma_3
\end{equation}
for some constant $\sigma_3$.
\end{proposition}

\begin{proof} The second identity of (\ref{eq31}), by taking into account the first one, becomes
$$\left(\frac{\kappa'}{\kappa}\right)'+\kappa^2-\frac{2\sigma_1^2}{\kappa^4}-\frac{\sigma_1 \sigma_2}{\kappa^2}=0.$$
We now obtain a first integral of this equation. Set $w=\log\kappa$, that is, $\kappa=e^w$. Since   $w'=\kappa'/\kappa$, we find
$$w''+e^{2w}-2\sigma_1^2 e^{-4w}-\sigma_1\sigma_2 e^{-2w}=0.$$
In order to solve this ODE, put $z=w'$ and consider $z=z(w)$. Because of $w''=\frac12\frac{d(z^2)}{dw}$, we have
$$\frac{d(z^2)}{dw}+2e^{2w}-4\sigma_1^2e^{-4w}-2\sigma_1\sigma_2 e^{-2w}=0.$$
An integration of this ODE leads to 
$$(w')^2=z^2=-e^{2w}-\sigma_1^2 e^{-4w}-\sigma_1\sigma_2 e^{-2w}-\sigma_3$$
for some constant $\sigma_3$. Since $w=\log\kappa$, 
$$\left(\frac{\kappa'}{\kappa}\right)^2+\kappa^2+\sigma_1^2\kappa^{-4}+\sigma_1\sigma_2\kappa^{-2}+\sigma_3=0,$$
 or equivalently,
$$\kappa'^2+\kappa^4+\sigma_3\kappa^2+\sigma_1^2\kappa^{-2}+\sigma_1\sigma_2=0.$$
This proves (\ref{eq32}).

For the converse of the proposition, let $\kappa(s)$ be a positive and non-constant solution of (\ref{eq32}) and put $\tau(s)=\sigma_1/\kappa^{2}(s)$.  Consider   $\alpha$  a curve parameterized by arc length with curvature $\kappa$ and torsion $\tau$. From (\ref{eq32}), it follows 
\begin{equation}\label{eq34}
\left(\frac{\kappa'_\alpha}{\kappa_\alpha}\right)^2+\kappa_\alpha^2+\sigma_1^2\kappa_\alpha^{-4}+\sigma_1\sigma_2\kappa_\alpha^{-2}+\sigma_3=0.
\end{equation}
Differentiating with respect to $s$ and since $\kappa'_\alpha\not=0$, we obtain
$$\left(\frac{\kappa'_\alpha}{\kappa_\alpha}\right)'+\kappa_\alpha^2-2\sigma_1^2\kappa_\alpha^{-4}-\sigma_1\sigma_2\kappa_\alpha^{-2}=0.$$
Since $\tau_\alpha=\sigma_1/\kappa_\alpha^2$, then 
$$\left(\frac{\kappa'_\alpha}{\kappa_\alpha}\right)'+\kappa_\alpha^2-2\tau_\alpha^2-\sigma_2\tau_\alpha=0.$$
This equation implies the first of (\ref{eq33}). Since $\sigma_1^2\kappa_\alpha^{-4}+\sigma_1\sigma_2\kappa_\alpha^{-2}=\tau_\alpha^{2}+\sigma_2\tau_\alpha=\Sigma_\alpha$, we derive from (\ref{eq34})  the second identity of (\ref{eq33}).
\end{proof}

\begin{remark} With the notation of Proposition \ref{pr34}, the generating curves $\alpha$, $\beta$ of a minimal translation surface $\Psi(s,t)=\alpha(s)+\beta(t)$ satisfy the conditions of Proposition \ref{pr36} with $\sigma_1=c_1$, $\sigma_2=c_2$ and  $\sigma_1=\bar{c}_1$, $\sigma_2=\bar{c}_2$, respectively. So we find
$$
R_\alpha^2+\kappa_\alpha^2+\tau_\alpha^2+c_2\tau_\alpha+c_3=0,\quad
 R_\beta^2+\kappa_\beta^2+\tau_\beta^2+\bar{c}_2\tau_\beta+\bar{c}_3=0,
$$
for some constants $c_3$ and $\bar{c}_3$.
 \end{remark}

 Motivated by equations (\ref{eq1}) and the set of identities (\ref{eq23}), we define the functions $V_i=V_i(s)$, $W_i=W_i(t)$, $1\leq i\leq 3$ by
 \begin{equation}\label{eq37}
 \left.
 \begin{split}
 V_1 &=\kappa_\alpha\b_\alpha\\
V_2 &=-\tau_\alpha\n_\alpha-R_\alpha\b_\alpha\\
V_3 &=\kappa_\alpha\t_\alpha-R_\alpha\n_\alpha+\frac{\Sigma_\alpha}{\tau_\alpha}\b_\alpha\\
\end{split}\right\},\quad  \left.
 \begin{split}
 W_1 &=\kappa_\beta\b_\beta\\
W_2 &=-\tau_\beta\n_\beta-R_\beta\b_\beta\\
W_3 &=\kappa_\beta\t_\beta-R_\beta\n_\beta+\frac{\Sigma_\beta}{\tau_\beta}\b_\beta.\\
\end{split}\right\}
\end{equation}

 It is not difficult to see that  $V_i$ an $W_i$  satisfy the following equations of Frenet type:
$$\left.\begin{array}{lll}
V'_1&=&\kappa_\alpha V_2\\
V'_2&=&-\kappa_\alpha V_1+\tau_\alpha V_3\\
V'_3&=&-\tau_\alpha V_2
\end{array}\right\}\qquad
\left.\begin{array}{lll}
W_1'&=&\kappa_\beta W_2\\
W_2'&=&-\kappa_\beta W_1+\tau_\beta W_3\\
W_3'&=&-\tau_\beta W_2,
\end{array}\right\}
$$
respectively.

Also it is immediate from (\ref{eq37})   that   their mixed products are
$$(V_1,V_2,V_3)=\kappa_\alpha^2\tau_\alpha=c_1,\quad   (W_1,W_2,W_3)=\kappa_\beta^2\tau_\beta=\bar{c}_1.$$
With the above notation,  the identity (\ref{eq1}) and the eight relations (\ref{eq23})  are written, respectively, as
\begin{equation}\label{eq39}
\left.
\begin{split}
&\langle V_1,\t_\beta\rangle=\langle\t_\alpha,W_1\rangle,\\
& \langle V_2,\t_\beta\rangle=\langle\n_\alpha,W_1\rangle\\
&\langle V_3,\t_\beta\rangle=\langle\b_\alpha,W_1\rangle,\\
& \langle V_1,\n_\beta\rangle=\langle\t_\alpha,W_2\rangle\\
&\langle V_1,\b_\beta\rangle=\langle\t_\alpha,W_3\rangle,\\
& \langle V_2,\n_\beta\rangle=\langle\n_\alpha,W_2\rangle\\
&\langle V_3,\n_\beta\rangle=\langle\b_\alpha,W_2\rangle,\\
& \langle V_2,\b_\beta\rangle=\langle\n_\alpha,W_3\rangle\\
&\langle V_3,\b_\beta\rangle=\langle\b_\alpha,W_3\rangle.
\end{split}\right\}
\end{equation}
All above facts and formulas are needed in order to prove the main results in the next section.

\section{Classification and construction results}\label{se4}

Let $S\subset\r^3$ be a minimal translation surface with parameterization $\Psi(s,t)=\alpha(s)+\beta(t)$, where we suppose $\alpha$ and $\beta$  parameterized by arc length. Motivated by the relations (\ref{eq37}), for each point $\alpha(s)$ and $\beta(t)$, we define  a set of  linear transformations $L_{\alpha(s)}, L_{\beta(t)}:\r^3\rightarrow\r^3$ with matrices
$$L_{\alpha(s)}=\left(\begin{array}{lll} 0&0&\kappa_\alpha(s)\\
0&-\tau_\alpha(s)&-R_\alpha(s)\\
\kappa_\alpha(s)&-R_\alpha(s)&\dfrac{\Sigma_\alpha}{\tau_\alpha}(s)
\end{array}
\right), \quad L_{\beta(t)}=\left(\begin{array}{lll} 0&0&\kappa_\beta(t)\\
0&-\tau_\beta(t)&-R_\beta(t)\\
\kappa_\beta(t)&-R_\beta(t)&\dfrac{\Sigma_\beta}{\tau_\beta}(t)
\end{array}
\right)$$
with respect to the basis $\{\t_\alpha(s),\n_\alpha(s),\b_\alpha(s)\}$ and $\{\t_\beta(t),\n_\beta(t),\b_\beta(t)\}$, respectively.

Since the matrix $L_{\alpha(s)}$ is symmetric with respect to an orthonormal basis, the linear map    $L_{\alpha(s)}$ is self-adjoint for all $s$. Its characteristic equation is
$$-\lambda^3+\left(\frac{\Sigma_\alpha}{\tau_\alpha}-\tau_\alpha\right)\lambda^2+(\Sigma_\alpha+R_\alpha^2+\kappa_\alpha^2)\lambda+\kappa_\alpha^2\tau_\alpha=0,$$
or, because of (\ref{eq88})  and (\ref{eq33}), $-\lambda^3+c_2\lambda^2-c_3\lambda+c_1=0$. 
Thus the real eigenvalues $\lambda_1$, $\lambda_2$ and $\lambda_3$ are three constants independent of $s$ and satisfy
\begin{equation}\label{eq41}
\left.
\begin{split}
\lambda_1+\lambda_2+\lambda_3&=c_2\\
\lambda_1\lambda_2+\lambda_1\lambda_3+\lambda_2\lambda_3&=c_3\\
\lambda_1\lambda_2\lambda_3&=c_1
\end{split}\right\}.
\end{equation}
Analogously the real eigenvalues $\mu_1$, $\mu_2$ and $\mu_3$ of $L_{\beta(t)}$ are   constants and satisfy
\begin{equation}\label{eq42}
\left.
\begin{split}
\mu_1+\mu_2+\mu_3&=\bar{c}_2\\
\mu_1\mu_2+\mu_1\mu_3+\mu_2\mu_3&=\bar{c}_3\\
\mu_1\mu_2\mu_3&=\bar{c}_1
\end{split}\right\}.
\end{equation}

\begin{remark} Since the cubic equation
$$-\lambda^3+c_2\lambda^2-c_3\lambda+c_1=0$$
has three real roots, its discriminant
$$\Delta=18c_1c_2c_3-4c_1c_2^3+c_2^2c_3^2-4c_3^3-27c_2^2$$
is non-negative. In the case where $\Delta=0$, the cubic equation has a multiple root.
\end{remark}

Now we prove the key property that all transformations   $L_{\alpha(s)}$ and $L_{\beta(t)}$ coincide for any $s$ and $t$.

\begin{proposition}\label{pr42}
Let  $\Psi(s,t)=\alpha(s)+\beta(t)$ be a minimal translation surface. Then  $L_{\alpha(s)}=L_{\beta(t)}$ for all $s\in I$ and $t\in J$.
\end{proposition}

\begin{proof} We shall prove that $L_{\beta(t)}$ is the adjoint of $L_{\alpha(s)}$ for any $s$ and $t$, and since $L_{\alpha(s)}$ is self-adjoint, we conclude $L_{\alpha(s)}=L_{\beta(t)}$, proving the result. 

Thus we need to show that $\langle L_{\alpha(s)}(v),w\rangle=\langle v,L_{\beta(t)}(w)\rangle$ for all $v,w\in\r^3$.
Let
$$v=a_1\t_\alpha(s)+a_2\n_\alpha(s)+a_3\b_\alpha(s),\ \ w=b_1\t_\beta(t)+b_2\n_\beta(t)+b_3\b_\beta(t),$$
where $a_i=a_i(s),b_i=b_i(t)\in\r$. Then
$$
\langle L_\alpha(v),w\rangle=\langle a_1L_\alpha(\t_\alpha)+a_2 L_\alpha(\n_\alpha)+a_3L_\alpha(\b_\alpha),b_1\t_\beta+b_2\n_\beta+b_3\b_\beta\rangle,$$
where we omit the dependence on $s$ and $t$. In the right hand side of this identity appears   nine summands. For each on them, we use the definition of $L_{\alpha}$ and the relations (\ref{eq39}). For instance,  we find for the first summand that
$$\langle L_\alpha(\t_\alpha),\t_\beta\rangle=\langle\kappa_\alpha\b_\alpha,\t_\beta\rangle=\langle\t_\alpha,\kappa_\beta\b_\beta\rangle=\langle\t_\alpha,L_\beta(\t_\beta)\rangle.$$
On applying this argument summand-by-summand, we see that
$$\langle L_\alpha(v), w\rangle=\langle v,b_1 L_\beta(\t_\beta)+b_2 L_\beta(\n_\beta)+b_3 L_\beta(\b_\beta)\rangle=\langle v,L_\beta(w)\rangle,$$
as desired. 
\end{proof}

 Because of $L_{\alpha(s)}=L_{\beta(t)}$ for all $s,t$, we conclude that the eigenvalues of $L_{\alpha(s)}$ and $L_{\beta(t)}$ coincide. Let $\lambda_1$, $\lambda_2$ and $\lambda_3$ denote the three eigenvalues. It follows from (\ref{eq41}) and (\ref{eq42}) that   $c_i=\bar{c}_i$, $1\leq i\leq3$.  Moreover, $L_{\alpha(s)}$, $L_{\beta(t)}$ have a common eigensystem independent of $s$ and $t$, for all $s\in I$, $t\in J$.

Consider now  the common eigensystem of $L_\alpha$ and $L_\beta$ as an orthonormal reference system. With respect to this system, we write $\alpha$ in coordinates, say,  $\alpha(s)=(\alpha_1(s),\alpha_2(s),\alpha_3(s))$, being $s$ the arc length. Then
$$\t_\alpha(s)=\alpha'(s)=(\alpha'_1(s),\alpha'_2(s),\alpha'_3(s)),$$
and
$$\kappa_\alpha(s)\b_\alpha(s)=L_\alpha(\t_\alpha(s))=(\lambda_1\alpha_1'(s),\lambda_2\alpha_2'(s),\lambda_3\alpha_3'(s)).$$
The identities $\langle\t_\alpha,\t_\alpha\rangle=1$ and $\langle\t_\alpha,\kappa_\alpha\b_\alpha\rangle=0$ write as
\begin{equation}\label{eq43}
\alpha'_1(s)^2+\alpha'_2(s)^2+\alpha'_3(s)^2=1
\end{equation}
and
\begin{equation}\label{eq44}
\lambda_1\alpha'_1(s)^2+\lambda_2\alpha'_2(s)^2+\lambda_3\alpha'_3(s)^2=0,
\end{equation}
respectively. Because of the third of (\ref{eq41}) and (\ref{eq44}), we conclude that all $\lambda_i$, $1\leq i\leq 3$, are non-zero and without the same sign. In the case where $c_1>0$, by renumbering the axis, we may choose $\lambda_1\leq\lambda_2<0<\lambda_3$.
Analogously, if $c_1<0$, we may choose $\lambda_1\geq\lambda_2>0>\lambda_3$. Set 
$$A=\sqrt{\frac{\lambda_3}{\lambda_3-\lambda_1}},\ B=\sqrt{\frac{\lambda_3}{\lambda_3-\lambda_2}}.$$
Taking into account (\ref{eq43}) and (\ref{eq44}), we may assume that
\begin{eqnarray*}
\alpha'_1(s)&=&A\cos w(s)\\
\alpha'_2(s)&=&B\sin w(s)\\
 \alpha'_3(s)&=&\sqrt{1-A^2\cos^2 w(s)-B^2\sin^2w(s)}.
 \end{eqnarray*}
Obviously, we have $\alpha'_3(s)\not=0$ everywhere. We briefly write
$$
\alpha'(s)=(A\cos w(s),B\sin w(s),\alpha'_3(s)).
$$
In order to calculate   $\kappa_\alpha$ and  $\tau_\alpha$, we need the computations of $\alpha''$ and $\alpha'''$: 
$$\alpha''=w'\left(-A\sin w,B\cos w,\frac{(A^2-B^2)\cos w\sin w}{\alpha'_3}\right),$$
\begin{eqnarray*}
 \alpha'''&=&w'' \left(-A\sin w,B\cos w,\frac{(A^2-B^2)\cos w\sin w}{\alpha'_3}\right)\\
&+&w'^2\left( 
-A\cos w,-B\sin w,\dfrac{(A^2-B^2)\left((1-A^2\cos^2 w)\cos^2 w-(1-B^2\sin^2w)\sin^2w\right)}{\alpha_3'^3}\right).
\end{eqnarray*}
Then
\begin{equation}\label{eq46}
\alpha'(s)\times\alpha''(s)=\frac{w'(s)}{\alpha'_3(s)}\left(B(A^2-1)\cos w(s), A(B^2-1)\sin w(s), AB\alpha'_3(s)\right)
\end{equation}
and

\begin{equation}\label{eq466}
(\alpha'(s),\alpha''(s),\alpha'''(s))=AB(1+A^2B^2-A^2-B^2)\left(\frac{w'(s)}{\alpha'_3(s)}\right)^3.
\end{equation}
Since $\tau_\alpha=(\alpha'(s),\alpha''(s),\alpha'''(s))/\kappa_\alpha^2$, it follows from  (\ref{eq31}) that  $(\alpha'(s),\alpha''(s),\alpha'''(s))=c_1$. The computation of the right hand side of (\ref{eq466})  using the third of (\ref{eq41}) yields 
$$\left(\frac{w'(s)}{\alpha'_3(s)}\right)^3=\left((\lambda_3-\lambda_1)(\lambda_3-\lambda_2)\right)^{3/2},$$
or, equivalently,
\begin{equation}\label{eq47}
\frac{w'(s)}{\alpha'_3(s)}=\sqrt{(\lambda_3-\lambda_1)(\lambda_3-\lambda_2)}.
\end{equation}

By a direct calculation from (\ref{eq46}), we obtain
$$\kappa_\alpha(s)^2=\left(\frac{w'(s)}{\alpha'_3(s)}\right)^2(A^2+B^2-A^2B^2-1+\alpha'_3(s))^2,$$
which can be written as
$$\kappa_\alpha(s)^2-w'(s)^2=(A^2+B^2-A^2B^2-1)\left(\frac{w'(s)}{\alpha'_3(s)}\right)^2.$$
Using the value of $A$ and $B$ and (\ref{eq47}), the above equation reduces into
 $\kappa_\alpha(s)^2-w'(s)^2=-\lambda_1\lambda_2$, hence
$$w'(s)=\sqrt{\kappa_\alpha(s)^2+\lambda_1\lambda_2}.
$$
Similarly, the above argument applies   to the curve $\beta$.

In the meantime, the curvatures $\kappa_\alpha(s)$, $\kappa_\beta(t)$  are positive solutions of the autonomous ODE
$$y'^2+y^4+c_3y^2+\frac{c_1^2}{y^{2}}+c_1c_2=0.$$
Using the value of $c_i$ from (\ref{eq41}), this equation is equivalent to
\begin{equation}\label{eq49}
y'^2+\frac{1}{y^2}(y^2+\lambda_1\lambda_2)(y^2+\lambda_1\lambda_3)(y^2+\lambda_2\lambda_3)=0.
\end{equation}
The positive equilibrium solutions are
$y_1=\sqrt{-\lambda_1\lambda_3}$, $y_2=\sqrt{-\lambda_2\lambda_3}$, which give stationary solutions of (\ref{eq49}). So the positive solutions $\kappa_\alpha(s)$ and $\kappa_\beta(t)$ are included in the   strip bounded by   the values $y_1$, $y_2$ and $\kappa_\beta(t)$ is a horizontal translate of $\kappa_\alpha(s)$. That is, $\kappa_\beta(t)=\kappa_\alpha(\pm t+c_0)$,    $c_0\in\r$. By a reparameterization of $\beta$, we conclude that $\kappa_\beta(t)=\kappa_\alpha(t)$ and thus $\tau_\beta(t)=\tau_\alpha(t)$ from (\ref{eq88}). Hence the generating curves $\alpha$ and $\beta$ are congruent.

Summarizing, we have proved the following classification result.

\begin{theorem}\label{t44} Let $\Psi(s,t)=\alpha(s)+\beta(t)$ be a minimal translation surface with $\alpha$ and $\beta$ parameterized by arc length. Suppose   $\kappa_\alpha, \kappa_\beta>0$  and  $\tau_\alpha\not=0, \tau_\beta\not=0$ everywhere. Then:
\begin{enumerate}
\item There are two constants $c_1, c_2\in\r$,  $c_1\not=0$, such that
$$\kappa_\alpha^2\tau_\alpha=\kappa_\beta^2\tau_\beta=c_1,\quad  \frac{\Sigma_\alpha}{\tau_\alpha}-\tau_\alpha=\frac{\Sigma_\beta}{\tau_\beta}-\tau_\beta=c_2.$$

\item The curvature $\kappa_\alpha$, $\kappa_\beta$ are positive solutions of the autonomous ODE
$$y'^2+y^4+c_3y^2+\frac{c_1^2}{y^{2}}+c_1c_2=0$$
for some constant $c_3$, and the curves $\alpha$ and $\beta$ have the same orbit.
\item Up to a rigid motion, we have
$$\alpha(s)=\left(A\int^s\cos w(s),B\int^s\sin w(s),\int^s\sqrt{1-A^2\cos^2 w(s)-B^2\sin^2w(s)}\right),$$
$$\beta(s)=\left(A\int^t\cos w(t),B\int^t\sin w(t),\int^t\sqrt{1-A^2\cos^2 w(t)-B^2\sin^2w(t)}\right),$$
where
$$A=\sqrt{\frac{\lambda_3}{\lambda_3-\lambda_1}},\ B=\sqrt{\frac{\lambda_3}{\lambda_3-\lambda_2}}$$
and $\lambda_1\leq\lambda_2<0<\lambda_3$ (resp. $\lambda_1\geq\lambda_2>0>\lambda_3$) are the real roots of the cubic equation
$-\lambda^3+c_2\lambda^2-c_3\lambda+c_1=0$ and
\begin{equation}\label{eqw}
w(s)=\int^s\sqrt{\kappa_\alpha(s)^2+\lambda_1\lambda_2}.
\end{equation}
\end{enumerate}
\end{theorem}

In the sequel, we will prove one more result which is the converse of Theorem \ref{t44} and, by the way, it provides an useful tool for constructing minimal translation surfaces.

\begin{theorem}\label{t45} Suppose that $c_1\not=0$, $c_2$ and $c_3$ are constants such that the cubic equation
\begin{equation}\label{eq50}
-\lambda^3+c_2\lambda^2-c_3\lambda+c_1=0
\end{equation}
has three real roots $\lambda_1$, $\lambda_2$, $\lambda_3$. Consider the autonomous ODE
\begin{equation}\label{eq51}
y'^2+y^4+c_3y^2+\frac{c_1^2}{y^2}+c_1c_2=0
\end{equation}
and let $\kappa_\alpha(s)=\kappa_\alpha(s;c_1,c_2,c_3)$ be a positive and non-constant solution of (\ref{eq51}). Denote by $\alpha(s)$ the curve parameterized by arc length $s$ with curvature $\kappa_\alpha(s)$ and torsion $\tau_\alpha(s)=c_1/\kappa_\alpha(s)^2$. Then the translation surface $\Psi(s,t)=\alpha(s)+\alpha(t)$ is minimal.
\end{theorem}

\begin{proof} The autonomous ODE (\ref{eq51}) takes the form (\ref{eq49}). Hence   $\lambda_i\not=0$ and without the same sign. In the case where $c_1>0$, we may choose $\lambda_1\leq\lambda_2<0<\lambda_3$ (analogously, if $c_1<0$, we choose $\lambda_1\geq\lambda_2>0>\lambda_3$). By the converse of Proposition \ref{pr36},  we deduce for the curve $\alpha$ that 
$$\kappa_\alpha^2\tau_\alpha=c_1,\quad \frac{\Sigma_\alpha}{\tau_\alpha}-\tau_\alpha=c_2,\quad  \Sigma_\alpha+R_\alpha^2+\kappa_\alpha^2=-c_3.
$$
At the point $\alpha(s)$, we define the linear transformation $L_{\alpha(s)}$ by the relations
\begin{equation}\label{eq53}
\left.
\begin{split}
L_{\alpha(s)}(\t_\alpha(s))&=\kappa_\alpha(s)\b_\alpha(s)\\
L_{\alpha(s)}(\n_\alpha(s))&=-\tau_\alpha(s)\n_\alpha(s)-R_\alpha(s)\b_\alpha(s)\\
L_{\alpha(s)}(\b_\alpha(s))&=\kappa_\alpha(s)\t_\alpha(s)-R_\alpha(s)\n_\alpha(s)+\frac{\Sigma_\alpha}{\tau_\alpha}(s)\b_\alpha(s)
\end{split}\right\}.
\end{equation}
The matrix of this transformation with respect to the basis $\{\t_\alpha(s),\n_\alpha(s),\b_\alpha(s)\}$ is symmetric and thus any $L_{\alpha(s)}$ is self-adjoint. The characteristic equation of $L_{\alpha(s)}$ is (\ref{eq50}) for any $s$. Moreover, by differentiation of (\ref{eq53}) with respect to $s$, and taking into account the Frenet equations, we find
$$L'_{\alpha(s)}(\t_\alpha(s))=0,\ L'_{\alpha(s)}(\n_\alpha(s))=0,\ L'_{\alpha(s)}(\b_\alpha(s))=0.$$
Thus, $L_\alpha=L_{\alpha(s)}$   is a constant transformation and has  a constant eigensystem for any $s$. Taking the eigensytem as the reference system as in Theorem \ref{t44}, we obtain
$\alpha'(s)=\left(A\cos w(s),B\sin w(s),\alpha'_3(s)\right),$
where
$$\alpha'_3(s)=\sqrt{1-A^2\cos^2w(s)-B^2\sin^2w(s)},$$
and $A=\sqrt{\lambda_3/(\lambda_3-\lambda_1)}$ and $B=\sqrt{ \lambda_3/(\lambda_3-\lambda_2)}$. Moreover,  
$$(\alpha'(s),\alpha''(s),\alpha'''(s))=\kappa_\alpha^2\tau_\alpha=c_1=\lambda_1\lambda_2\lambda_3=AB(1+A^2B^2-A^2-B^2)\left(\frac{w'(s)}{\alpha'_3(s)}\right)^3,$$
 hence
$$\kappa_\alpha(s)^2-w'(s)^2=-\lambda_1\lambda_2.$$
We now prove that the surface $\Psi(s,t)=\alpha(s)+\alpha(t)$ is minimal. The condition $H=0$ in    (\ref{eq2}) is now
$\langle\alpha'(s)\times \alpha''(s),\alpha'(t)\rangle=\langle \alpha'(s),\alpha'(t)\times \alpha''(t)\rangle$. The computation of both Euclidean products give
\begin{eqnarray*}
&&\langle\alpha'(s)\times \alpha''(s),\alpha'(t)\rangle=\\
&&\frac{AB w'(s)}{\alpha_3'(s)}\left(
(A^2-1)\cos w(s)\cos w(t)+(B^2-1)\sin w(s)\sin w(t)+\alpha_3'(s)\alpha_3'(t)\right).
\end{eqnarray*}
\begin{eqnarray*}
&&\langle\alpha'(s),\alpha'(t)\times\alpha''(t)\rangle=\\
&&\frac{AB w'(t)}{\alpha_3'(t)}\left(
(A^2-1)\cos w(s)\cos w(t)+(B^2-1)\sin w(s)\sin w(t)+\alpha_3'(s)\alpha_3'(t)\right).
\end{eqnarray*}
Thus the surface is minimal if and only if we prove that  $w'(s)/\alpha_3'(s)=w'(t)/\alpha_3'(t)$ for all $s$ and $t$. However this holds  because of (\ref{eq47}), we deduce 
 $$\frac{w'(s)}{\alpha'_3(s)}=\sqrt{(\lambda_3-\lambda_1)(\lambda_3-\lambda_2)}=\frac{w'(t)}{\alpha_3'(t)}.$$
\end{proof}

\begin{remark}
 If the characteristic equation (\ref{eq50}) has a double root, that is, $\lambda_1=\lambda_2$, then $A=B=\sqrt{\lambda_3/(\lambda_3-\lambda_1)}$, $\alpha'(s)=
(A\cos w(s),A\sin w(s),\sqrt{1-A^2})$ and $\kappa_\alpha=A^2w'^2$. So, because of $w'^2=\kappa_\alpha^2+\lambda_1\lambda_2$, we see that  $\kappa_\alpha^2=-\lambda_1\lambda_3$ and $\tau_\alpha=-\lambda_1$. Since the curvature and torsion are constant,  the curve $\alpha$ is a circular helix. On the other hand,  the autonomous ODE (\ref{eq51}) becomes
$$y'^2+\frac{1}{y^2}(y^2+\lambda_1^2)(y^2+\lambda_1\lambda_3)^2=0,$$
we conclude that there are no non-constant solutions of (\ref{eq51}).
\end{remark}

 We finish this paper showing explicit examples of  the  procedure for constructing translation minimal surfaces with non-planar generating curves according Theorem \ref{t45}. In a first step, and looking for examples of minimal translation surfaces, recall that by the item 1 of Theorem \ref{t44},  if the generating curve $\alpha$ has constant curvature (resp. constant torsion), then its torsion (resp. curvature) is constant as well, hence the curve is a  circular helix and the resulting surface is a helicoid by Theorem \ref{t32}.

 Fixing the constants $c_i$ is equivalent to fix the roots $\lambda_i$ of the cubic polynomial  (\ref{eq50}).
\begin{remark} The family of minimal translation surfaces is constructed in terms of the roots of the cubic polynomial equation $-\lambda^3+c_2\lambda^2-c_3\lambda+c_1=0$. After a homothety of the ambient space $\r^3$, which preserves the minimality of the surface and the property to be a translation surface, we can fix one of the roots of this equation. As a consequence, the minimal translation surfaces is parameterized by two parameters.
\end{remark}

Following Theorems \ref{t44} and \ref{t45}, we present here the scheme for constructing examples of minimal translation surfaces in Euclidean space. 
\begin{enumerate}
\item[{\bf Step 1:}] Fix the roots $\lambda_i$ of (\ref{eq50}). By simplicity, we may  consider $\lambda_1\leq\lambda_2<0<\lambda_3$. The root $\lambda_3$ will be fixed to be $\lambda_3=1$. Compute $A$, $B$.
\item[{\bf Step 2:}]  Compute $c_i$ and the polynomial equation (\ref{eq50}).
\item[{\bf Step 3:}]  Compute the equilibrium points $y_1=\sqrt{-\lambda_2\lambda_3}$ and $y_2=\sqrt{-\lambda_1\lambda_3}$ of (\ref{eq51}).
\item[{\bf Step 4:}]  Fix $y_0$ the initial value of (\ref{eq51}), where $y_1<y_0<y_2$.
\item[{\bf Step 5:}]  Solve numerically the equations (\ref{eq51}). Fix a initial value $w_0$ to solve numerically   (\ref{eqw}) and the function $w$.
\item[{\bf Step 6:}]  Solve the curve $\alpha$.
\end{enumerate}

{\bf Example 1.} Case of helicoid. Choose a double root $\lambda_1=\lambda_2=-1$. Then (\ref{eq50}) is $-\lambda^3-\lambda^2+\lambda+1=0$ and $A=B=1/\sqrt{2}$. The equilibrium points as $y_1=y_2=1$. Thus take $y_0=1$ as initial condition in (\ref{eq51}). Then the solution is $\kappa(s)=1$, so $\tau=1$.

\begin{figure}[hbtp]\begin{center}
\includegraphics[width=.7\textwidth]{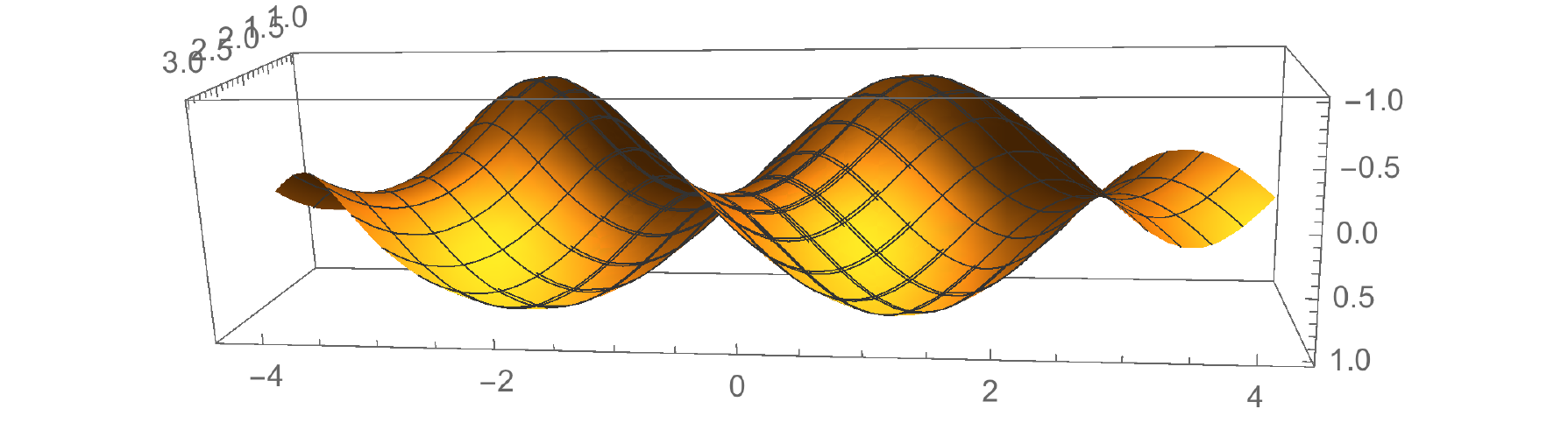}
\end{center}
\caption{The helicoid}
\end{figure}

{\bf Example 2.} Take  $\lambda_1=-4$ and $\lambda_2=-1$. Then (\ref{eq50}) is $-\lambda^3-4\lambda^2+\lambda+4=0$ and $A=0.447$ and $B=0.707$. The equilibrium points as $y_1=1$ and $y_2=2$. Choose $y_0=1.3$ as initial condition in (\ref{eq51}).
\begin{figure}[hbtp]\begin{center}
\includegraphics[width=.7\textwidth]{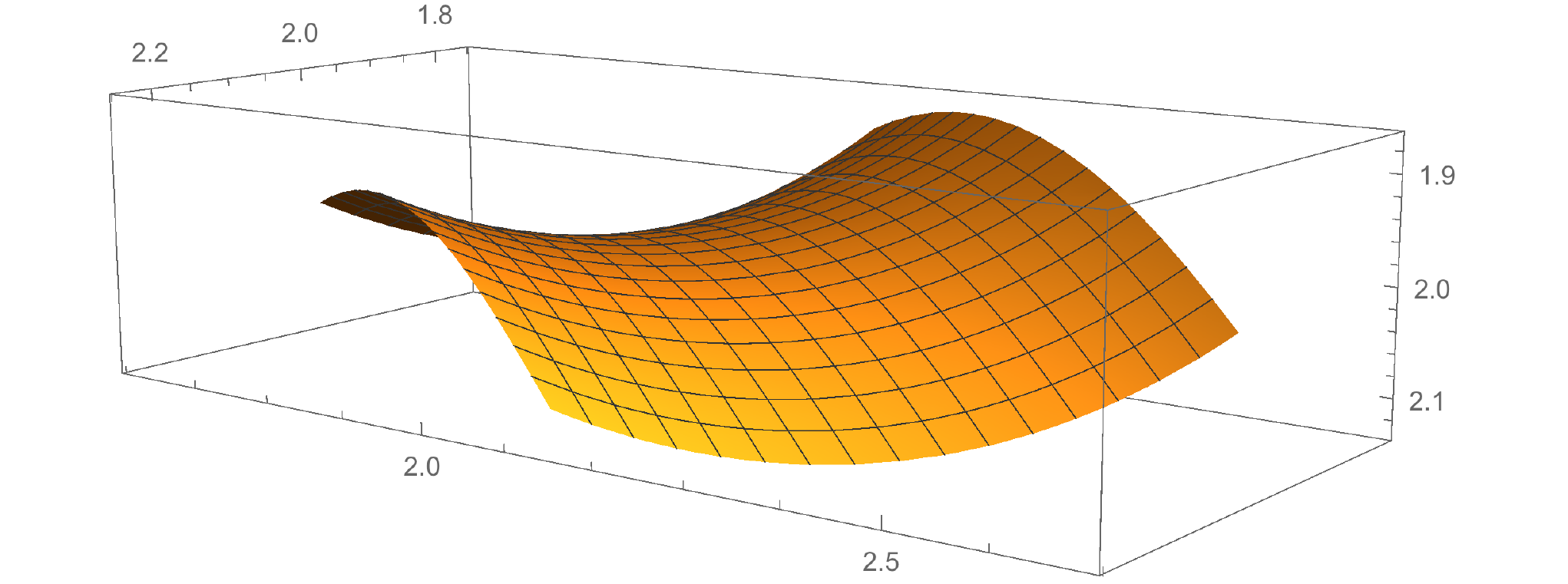}
\end{center}
\caption{Case $\lambda_1=-4$, $\lambda_2=-1$ and $\lambda_3=1$}
\end{figure}

{\bf Example 3.} Consider $\lambda_1=-2$ and $\lambda_2=-1$.  Then the polynomial is $p[\lambda]=-\lambda^3-2\lambda^2+\lambda+2$. The equilibrium points as $y_1=1.412$ and $y_2=1$. Also $A=0.577$ and $B=0.707$. The initial value is $y_0=1.1$.

\begin{figure}[hbtp]\begin{center}
\includegraphics[width=.7\textwidth]{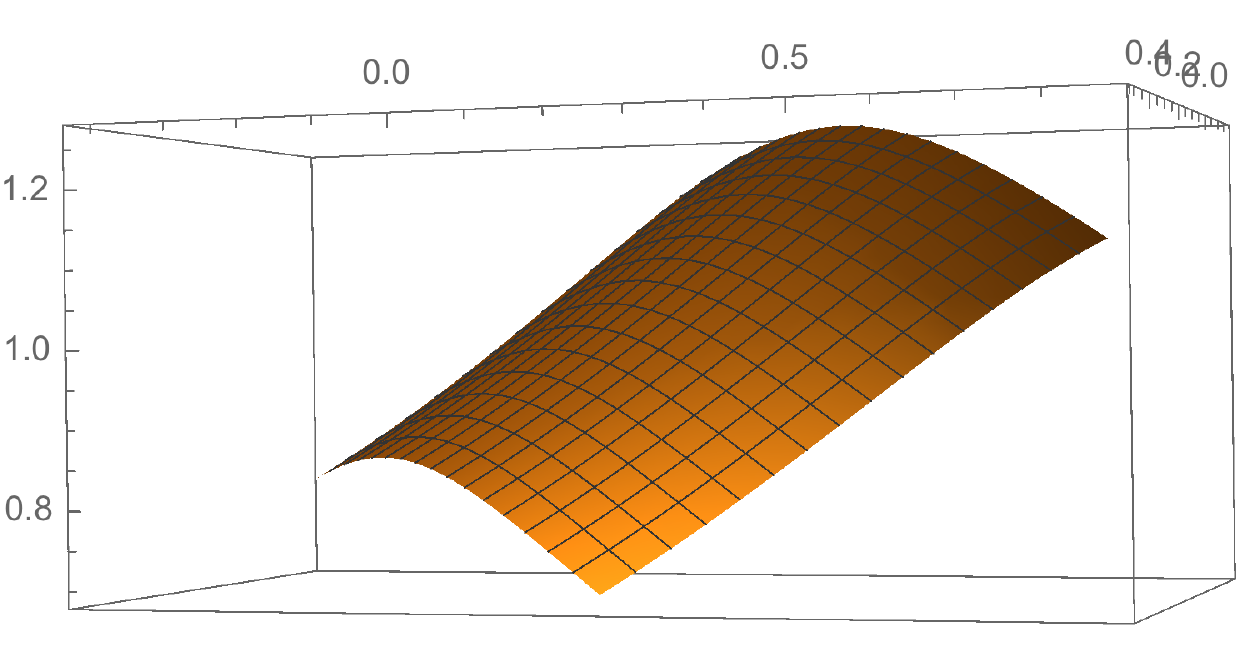}
\end{center}
\caption{Case $\lambda_1=-2$, $\lambda_2=-1$ and $\lambda_3=1$}\end{figure}


\end{document}